\documentclass[12pt]{amsart}
\newtheorem{theorem}{Theorem}[section]

\newtheorem{definition}[theorem]{Definition}
\newtheorem{es}[theorem]{Example}
\newtheorem{problem}[]{Problem}

\def\phi{\varphi}

\title[Differential game on a convex set in the
plane]{Differential game of many pursuers with integral
constraints on a convex set in the plane}
\author[I. A. Alias]{Idham Arif Alias}
\address[Idham Arif Alias]{Department of Mathematics and Institute for Mathematical Research, Universiti Putra Malaysia, Serdang, Malaysia} \email{idham\_aa@upm.edu.my}
\author[G. Ibragimov]{Gafurjan Ibragimov}
\address[Gafurjan Ibragimov]{Department of Mathematics and Institute for Mathematical Research, Universiti Putra Malaysia, Serdang, Malaysia} \email{ibragimov@upm.edu.my}
\author[M. Ferrara]{Massimiliano Ferrara}
\address[Massimiliano Ferrara]{Department of Law and Economics, University Mediterranea of Reggio Calabria,
Italy, and CRIOS, Center for Research in Innovation, Organization
and Strategy, Department of Management and Technology, Bocconi
University, Italy} \email{massimiliano.ferrara@unirc.it
and[@unibocconi.it]}
\author[M. Salimi]{Mehdi Salimi}
\address[Mehdi Salimi]{Center for Dynamics, Department of Mathematics, Technische Universit{\"a}t Dresden, Germany, and MEDAlics, Research Centre at
the University Dante Alighieri Reggio Calabria, Italy}
\email{mehdi.salimi@tu-dresden.de and[@medalics.org]}
\author[M. Monsi]{Mansor Monsi}
\address[Mansor Monsi]{Department of Mathematics and Institute for Mathematical Research, Universiti Putra Malaysia, Serdang, Malaysia} \email{mmonsi@upm.edu.my}
\begin{document}
\begin{abstract}
 We study a simple motion differential game of many pursuers and one evader in the plane. We give a nonempty closed convex set in the plane, and the pursuers and evader move on this set. They cannot leave this set during the game. Control functions of players are subject to coordinate-wise integral constraints. If the state of the evader $y$, coincides with that of a pursuer $x_i$, $i=\{1,...,m\}$, at some time $t_i$ (unspecified), i.e. $x_i(t_i)=y(t_i)$, then we say that pursuit is completed. We obtain some conditions under which pursuit can be completed from any position of the players in the given set. Moreover, we construct strategies for the pursuers.
\end{abstract}
\maketitle

\section{Introduction}

Differential game comes into play to study procedures in which one
controlled object will be pursued by others. There are several
types of differential games, but the most common one is the so
called pursuit-evasion game. Significant researches by Breitner et
al. \cite{bre}, Chikrii \cite{chikrii}, Friedman \cite{1}, Isaacs
\cite{isa}, Krasovskii \cite{kra}, Petrosyan \cite{pet},
Pontryagin \cite{pon} and Rzymowski \cite{rzy} deal with
differential games and pursuit-evasion problems.

Satimov et al.\ study a linear pursuit differential game of many
pursuers and one evader with integral constraints on controls of
players in the space ${\rm I\!R}^n$ \cite{10}. In his paper, the
objects move according to the equations
$$\dot{z}_{i}=C_{i}z_{i}-u_{i}+v, \ \ \ z_i(t_0)=z_i^0, \ \ \ i=1,...,m$$
where $u_i$ and $v$   are the control parameters of the $i$-th
pursuer and the evader, respectively. One of the hypotheses of the
main result of his research work is that the eigenvalues of the
matrices $C_i$ are real. It was shown that if the total resource
of controls of the pursuers is greater than that of the evader,
then under certain conditions pursuit can be completed.

Differential game of one pursuer and one evader with integral
constraints studied by Ibragimov \cite{19} occurs on a closed
convex subset $S$ of ${\rm I\!R}^n$ and dynamics of the players
are described by the equations
$$\dot{x}=\alpha(t)u, \ \ \ x(0)=x_0, \ \ \ \int_{0}^{\infty}|u(s)|^2ds \leq \rho^2 $$ \
$$\dot{y}=\alpha(t)v, \ \ \ y(0)=y_0, \ \ \ \int_{0}^{\infty}|v(s)|^2ds \leq \sigma^2.$$
In the paper, evasion and pursuit problems were investigated and a
formula for optimal pursuit time was found and optimal strategies
of the players were constructed.

In the paper \cite{16}, a pursuit differential game of $m$
pursuers and $k$ evaders with integral constraints described by
the systems of differential equations
$$\dot{z}_{ij}=C_{ij}z_{i}+u_{i}-v_{j}, \ \ \ z_{ij}(t_0)=z_{ij}^{0}, \ \ \ i=1,...,m \ \ \ j=1,...,k,$$ \
$$\int_{0}^{\infty}|u_i(s)|^2ds \leq \rho_i^2 \ \ i=1,...,m \ \ \  \int_{0}^{\infty}|v_j(s)|^2ds\leq \sigma_j^2,\ \ \ j=1,...,k$$
where $u_i$ is the control parameter of the $i$-th pursuer, and $v_j$ is that of the $j$-th evader, was examined. Here, the eigenvalues of matrices  $C_{ij}$ are not necessarily real, and moreover, the number of evaders can be any.  Under assumption that the total resource of controls of the pursuers is greater than that of the evaders, that is

$$\rho_{1}^2+\rho_{2}^2+...+\rho_{m}^2>\sigma_{1}^2+\sigma_{2}^2+...
+\sigma_{k}^2,$$ and real parts of all eigenvalues of the matrices
$C_{ij}$ are nonpositive,  it was proved that pursuit can be
completed from any initial position.

In \cite{ibr2} Ibragimov and Salimi investigate a differential
game for inertial players with integral constraints under the
assumption that the control resource of the evader is less than
that of each pursuer. Ibragimov et al.\ in \cite{ibr3} investigate
an evasion from many pursuers in simple motion differential games
with integral constraints as well. Recently, in \cite{salimi}
Salimi et al.\ study a differential game in which countably many
dynamical objects pursue a single one. All the players perform
simple motions and the control of a group of pursuers are subject
to integral constraints and the control of the other pursuers and
the evader are subject to geometric constraints. They construct
optimal strategies for players and find the value of the game.

In the present paper, we consider a differential game of pursuers
and one evader with coordinate-wise integral constraints. Game
occurs in a nonempty closed convex set in the plane. We obtain
sufficient conditions of completion of the game. We attempt to
improve readers learning in control problems by constructing
strategies for the pursuers.

\section{Statement of the problem}

We consider a differential game described by the equations
\begin{equation}\label{(1)}
\dot{x}_i=u_i, \quad x_i(0)=x_{i0}, \quad i=1,...m, \quad
\dot{y}=v, \quad y(0)=y_0,
\end{equation}
where $x_{i}$, $u_{i}$, $y$, $v\in{\rm I\!R}^{2}$, \, $u_{i}$ is
control parameter of the pursuer $x_i$\, $i=1,...,m$ and $v$ is
that of the evader $y$.

\begin{definition} A measurable function $u_i(t)=(u_{i1}(t),u_{i2}(t))$\, $t\geq 0$, is called admissible control of the pursuer $x_i$ if
\begin{equation}\label{(2)}
\int_{0}^{\infty}|u_{ij}(s)|^2ds \leq \rho_{ij}^2,
\end{equation}
where $\rho_{ij}$, $i=1,...,m$, $j=1,2$, are given positive numbers.
\end{definition}

\begin{definition} A measurable function $v(t)=(v_{1}(t),v_{2}(t))$\, $t\geq 0$, is called admissible control of the evader if
\begin{equation}\label{(3)}
\int_{0}^{\infty}|v_{j}(s)|^2ds \leq \sigma_{j}^2,
\end{equation}
where $\sigma_{j}$, $j=1,2$ are given positive numbers.
\end{definition}

\begin{definition}\label{(def)} A Borel measurable function $U_i(x_i,y,v)=(U_{i1}(x_i,y,v),\\ U_{i2}(x_i,y,v))$, $U_i:{\rm I\!R}^6\rightarrow {\rm I\!R}^2$ is called a strategy of the pursuer $x_i$ if for any control of the evader $v(t)$, $t\geq 0$, the initial value problem
\begin{equation*}
\dot{x_i}=U_i(x_i, y, v(t)), \ \ x_i(0)=x_{i0}, \quad
\dot{y}=v(t), \ \ y(0)=y_0,
\end{equation*}
has a unique solution $(x_{i}(t),y(t))$ and the inequalities
$$\int_0^\infty |U_{i1}(x_i(s),y(s),v(s))|^2ds\leq \rho_{i1}^2, \quad \int_0^\infty |U_{i2}(x_i(s),y(s),v(s))|^2ds\leq \rho_{i2}^2$$
hold.
\end{definition}

\begin{definition} We say that pursuit can be completed from the initial position $\{x_{10},...x_{m0},y_0\}$ for the time $T$ in the game (\ref{(1)})-(\ref{(3)}), if there exist strategies $U_i$, $i=1,...m$, of the pursuers such that for any control $v=v(\cdot)$ of the evader the equality
$x_i(t)=y(t)$ holds for some $i\in\{1,...,k\}$  and $t\in[0, T]$.
\end{definition}

Given a nonempty closed convex set $N \subset {\rm I\!R}^n$,
according to the rule of the game, all players must not leave the
set $N$. This information describes a differential game of many
players with integral constraints on control functions of players.

\begin{problem}
Find a sufficient condition of completion of pursuit in the game
(\ref{(1)})-(\ref{(3)}).
\end{problem}

\section{Main result}

Now we formulate the main result of the paper.
\begin{theorem}\label{(the)} If the inequality
\begin{equation}\label{(4)}
\rho_{1j}^{2}+\rho_{2j}^{2}+...+\rho_{mj}^{2}>\sigma_{j}^{2}
\end{equation}
holds for a $j\in\{1,2\}$, then pursuit can be completed for a
finite time   $T$ in the game (\ref{(1)})-(\ref{(3)}) from any
initial position.
\end{theorem}

\begin{proof}
We prove the theorem when the equality (\ref{(4)}) holds at $j=1$.
The same reasoning applies to the case $j=2$. Thus,
$$\rho_{11}^{2}+\rho_{21}^{2}+...+\rho_{m1}^{2}>\sigma_{1}^{2}$$
and $\rho_{i2}, \ i = 1,2,..., m$, are positive numbers.
Denote
$$\sigma_{i1}=\frac{\sigma_1}{\rho_1}\rho_{i1},\, i=1,2,...,m; \ \
\rho_1=(\rho_{11}^2+\rho_{21}^2+...+\rho_{m1}^2)^{1/2}.$$ Clearly,
$\sigma_{i1}<\rho_{i1}$. Let $x_0$ and $y_0$ are points of the set
$N$ such that $|x_0-y_0|= diam \ N := \max_{x,y\in N}|x-y|$. We pass
the $x$-axis through the points $x_0$ and $y_0$. Let
$d=|x_0-y_0|$, $c=\max_{(\xi, \eta)\in N}|\eta|$ . Without any
loss of generality, we assume that all pursuers at the beginning
are on the $x$-axis, since otherwise we can bring to such position
by applying the following controls:

\begin{equation}\label{4a}
u_1(t)=0, \ \ u_2(t)=-\frac{x_{i2}^0}{T}, \ \ 0\leq t\leq T; \ \ T = \max_{i\in\{1,...,m\}}\frac{4|x_{i2}^0|^2}{\rho_{i2}^2},
\end{equation}
and then we consider the game with $\rho_{i2}'=(\sqrt{3}/2)\rho_{i2}$ instead of $\rho_{i2}$.

The strategy of each pursuer is constructed in two stages. In the
first stage, the pursuer moves with constant speed along $x$-axis
to become on one vertical line with the evader. In the second
stage, the pursuer uses P-strategy.

Let us construct a strategy for the pursuer $x_i$, $i=1,...,m$. This pursuer moves according to the strategy to be constructed on an interval $[\theta_i, \theta_{i+1}]$; $\theta_1=0$. We define $\theta_{i+1}$ inductively below. Set

\begin{equation}\label{(5)}
\begin{split}
u_{i1}(t)&=sgn(y_1(\theta_i)-x_{i1}(\theta_i))\frac{d}{t_{i1}},\\
\quad u_{i2}(t)&=0, \quad \theta_i<t\leq\theta_i+\tau_{i1}, \quad
i=1,...,m,
\end{split}
\end{equation}where $\tau_i$, $0 \leq \tau_i\leq t_{i1}$, is a time for which
$x_{i1}(\theta_i+\tau_{i1})=y_{1}(\theta_i+\tau_{i1})$, $t_{i1} = \frac{d^2}{\rho_{i1}^2 - \sigma_{i1}^2}$. Clearly,
such a time $\tau_{i1}$ exists since
$$
\left|\int\limits_{\theta_i}^{\theta_i+t_{i1}}u_{i1}(t)dt \right| = t_{i1}\cdot \frac{d}{t_{i1}} = d,
$$
that is on the time interval $[\theta_i, \theta_i + t_{i1}]$ the pursuer $x_i$ can travel the distance equal to $d$.

We now construct the second part of the strategy of the pursuer $x_i$. Set

\begin{equation}\label{(6)}
\begin{split}
&u_{i1}(t)=v_1(t), \quad u_{i2}(t)=sgn(y_2(\theta_i+\tau_{i1})-x_{i2}(\theta_i+\tau_{i1}))\frac{c}{t_{i2}},\\
&\theta_i+\tau_i<t\leq \theta_{i+1}, \quad \text{if} \quad \
x_i(t)\in int \ N,\\
&\text{and}\\
&u_{i1}(t)=0, \ \ u_{i2}(t)=0, \quad \text{if} \quad  x_i(t)\in
\partial N,
\end{split}
\end{equation}
where $\partial N$ is boundary of the set $N$, $int \ N$ is
interior of $N$, \ $t_{i2}=\frac{c^2}{\rho_{i2}^2}$, \
$\theta_{i+1}=\theta_i+\tau_{i1}+t_{i2}$. On the time interval
$[\theta_{i},\theta_{i+1}]$  all the other pursuers $x_j$,
$j\in\{1,...,i-1,i+1,...m\}$ do not move.  In the other words,
$$u_i(t)\equiv0, \ \ j\in\{1,...,i-1,i+1,...m\}, \ \ t\in [\theta_i, \theta_{i+1}].$$
Note that
$$x_{i1}(t)=y_1(t), \ \ \theta_i+\tau_{i1}\leq t \leq \theta_{i+1},$$
whenever
\begin{equation}\label{(7)}
\int_{\theta_{i}+\tau_{i1}}^{\theta_{i+1}}v_1^2(t)dt\leq
\sigma_{i1}^2,
\end{equation}
since under this condition we have
\begin{equation*}
\begin{split}
\int_{\theta_{i}}^{\theta_{i+1}}u_{i1}^2(t)dt&=
\int_{\theta_{i}}^{\theta_i+\tau_{i1}}u_{i1}^2(t)dt+
\int_{\theta_i+\tau_{i1}}^{\theta_{i+1}}u_{i1}^2(t)dt\\
&=\frac{a^2}{t_{i1}}+
\int_{\theta_i+\tau_{i1}}^{\theta_{i+1}}v_1^2(t)dt\\
&\leq \rho_{i1}^2-\sigma_{i1}^2+\sigma_{i1}^2=\rho_{i1}^2.
\end{split}
\end{equation*}
Therefore, if (\ref{(7)}) holds then the pursuer $x_i$ is able to
apply the strategy (\ref{(6)}) that ensures the equality
$$x_{i2}(\theta_i+\tau_{i1}+\tau_{i2})=y_2(\theta_i+\tau_{i1}+\tau_{i2})$$
for some $0\leq \tau_{i2}\leq t_{i2}$. If (\ref{(7)}) fails to
hold for a control of the evader $v(t)=(v_1(t), v_2(t))$, then
pursuit may not be completed by the pursuer $x_i$. However,
(\ref{(7)}) holds true at least for one index $i\in \{1,...,m\}$.
Indeed, if
$$\int_{\theta_i+\tau_{i1}}^{\theta_{i+1}}v_1^2(t)dt>\sigma_{i1}^2$$
for all $i\in \{1,...,m\}$, then we have
$$ \int_{0}^{\theta_{m+1}}v_1^2(t)dt=\sum_{i=1}^{m}
\int_{\theta_i}^{\theta_{i+1}}v_1^2(t)dt\geq
\sum_{i=1}^{m}\int_{\theta_i+\tau_{i1}}^{\theta_{i+1}}v_1^2(t)dt>
\sum_{i=1}^{m}\sigma_{i1}^{2}=\sigma_1^2,$$ which is a
contradiction. Thus the inequality (\ref{(7)}) holds at least for
one index $i$ and then pursuit is completed by the pursuer $x_i$.
And the proof is complete.
\end{proof}

\section{Discussion and conclusion}

We have obtained a sufficient condition (\ref{(4)}) to complete
the pursuit in the differential game of $m$ pursuers and one
evader with coordinate-wise integral constraints. We have
constructed strategies of the pursuers and showed that pursuit can
be completed from any initial position in $N$.

Let us discuss information which is used by the pursuer. To bring
positions of all pursuers to $x$-axis by the control (\ref{4a}),
the pursuers use only their initial positions. Next, the pursuer
$x_i$ using $y_1(\theta_i)$ and $x_{i1}(\theta_i)$ constructs
$u_i(t)$ at the time $\theta_i$ by formula (\ref{(5)}) and uses it
on $(\theta_i, \theta_i + \tau_{i1}]$. While using (\ref{(5)}) the
pursuer $x_i$ observes whether $x_{i1}(t) = y_1(t)$. This is
possible since by Definition \ref{(def)}, the pursuer $x_i$ is
allowed to know  $x_{i1}(t)$ and $y_1(t)$. Finally,  as soon as
$x_{i1}(t) = y_1(t)$ at some time $t = \theta_i + \tau_{i1},$ the
pursuer $x_i$ uses (\ref{(6)}), which requires to know $v_1(t)$.
By Definition \ref{(def)}, the pursuer $x_i$ knows $v(t)$, in
particular, knows $v_1(t)$ as well.

\begin{es}
Let $m=2, \ \rho_{11} = 1, \ \rho_{21} = 1.1, \ \sigma_1 =
\sqrt{2}$ and $\rho_{12}, \ \rho_{22}$ be any positive numbers and
$$
N = \left\{(x,y)\left| \ \ \frac{x^2}{9} + \frac{y^2}{4} \le 1  \right.\right\}.
$$
Then one can see that $d = 6$, and $c = \max\{|y| \ ; \ (x,y) \in
N\} = 2$ and by theorem \ref{(the)} pursuit can be completed.
\end{es}

\section{Acknowledgement}

This research was supported by the Research Grant of the
Universiti Putra Malaysia, No. 05-02-12-1868RU, and by MEDAlics,
the Research Centre at the University Dante Alighieri, Reggio
Calabria, Italy.


\begin{thebibliography}{99}

\bibitem{bre} Breitner, M.H., Pesch, H.J., Grimm, W., Complex differential games of pursuit-evasion type with state
constraints, part 1: Necessary conditions for optimal open-loop
strategies, Journal of Optimization Theory and Applications,
78(3), 419-441, (1993)

\bibitem{chikrii} Chikrii, A.A., Differential Games with Many Pursuers, Mathematical
Control Theory. Banach Center Publ., PWN, Warsaw, 14, 81-107,
(1985)

\bibitem{1} Friedman, A., Differential Games, John Wiley \& Sons, New York, NY, USA,
(1971)

\bibitem{19} Ibragimov, G.I., A Game Problem on a Closed Convex Set, Siberian Advances in Mathematics, 12(3), 16-31, (2002)

\bibitem{16} Ibragimov, G.I., An n-person differential game with integral constraints on the controls of the players,
Russian Mathematics (Izvestiya VUZ. Matematika), 48(1), 45-49,
(2004)

\bibitem{ibr2} Ibragimov, G.I., Salimi, M., Pursuit-evasion differential game with
 many inertial players, Mathematical Problems in Engineering, vol. 2009, Article ID
653723, 15 pages, (2009)

\bibitem{ibr3} Ibragimov, G.I., Salimi, M., Amini, M., Evasion from many pursuers
in simple motion differential game with integral constraints,
European Journal of Operational Research, 218, 505-511, (2012)

\bibitem{isa} Isaacs, R.,  Differential Games, John Wiley \& Sons, New York, NY, USA,
(1965)

\bibitem{kra} Krasovskii, N.N., Control of a Dynamical System, Nauka, Moscow,
(1985)

\bibitem{pet} Petrosyan, L.A., Differential Pursuit Games, Izdat. Leningrad. Univ., Leningrad,
(1977)

\bibitem{pon} Pontryagin, L.S., Selected Works,  Moscow: MAKS Press.,
(2004)

\bibitem{rzy} Rzymowski, W., Evasion along each trajectory in differential games with many pursuers,
 Journal of Differential Equations, 62(3), 334-356, (1986)

\bibitem{salimi} Salimi, M., Ibragimov, G.I., Siegmund, S., Sharifi, S., On a fixed duration pursuit differential game with geometric and integral
 constraints, Dynamic Games and Applications, under revision, (2015)

\bibitem{10} Satimov, N.Yu., Rikhsiev, B.B., Khamdamov, A.A.,
 On a pursuit problem for n-person linear differential and discrete games with integral constraints, Mathematics of the USSR-Sbornik,
  46(4), 459-471, (1983)

\end{thebibliography}
\end{document}